\documentclass[a4paper]{article}
\usepackage{amssymb}
\usepackage{graphicx,color}
\usepackage{amsthm}
\usepackage{amsmath}
\usepackage{enumerate}
\usepackage{tikz}

\usepackage{float}
\usepackage{subfig}
\usepackage{hyperref}
\usepackage{multirow}
\usepackage{rotating}

\theoremstyle{plain}
\newtheorem{theorem}{Theorem}
\newtheorem{proposition}[theorem]{Proposition}
\newtheorem{corollary}[theorem]{Corollary}
\newtheorem{lemma}[theorem]{Lemma}

\theoremstyle{definition}

\newtheorem{definition}[theorem]{Definition}

\def\ra{\rightarrow}

\def\ba{\begin{array}}
\def\ea{\end{array}}
\def\bi{\begin{itemize}}
\def\ei{\end{itemize}}
\def\mR{\mathbb{R}}
\def\mZ{\mathbb{Z}}

\def\mC{\mathbb{C}}
\def\m1{1}

\let\oldhat\hat

\renewcommand{\hat}[1]{\oldhat{\mathbf{#1}}}

\providecommand{\com[2]}{{\color{blue} \begin{tt}[#1: #2]\end{tt}}}

\begin{document}
\title{Improved mixing rates of directed cycles by added connection}
\author{Bal\'azs Gerencs\'er\thanks{B. Gerencs\'er is with MTA
    Alfr\'ed R\'enyi Institute of Mathematics, Hungary,
    {\tt\small gerencser.balazs@renyi.mta.hu} He is supported by NKFIH
    (National Research, Development and Innovation Office)
    grant PD 121107.
    This work has been carried out during his stay at Universit\'e catholique de Louvain, Belgium.}%
  \and Julien M. Hendrickx\thanks{J. M. Hendrickx is with ICTEAM Institute,
    Universit\'e catholique de Louvain, Belgium,
    {\tt\small
      julien.hendrickx@uclouvain.be} 
    The work is supported by the DYSCO Network (Dynamical Systems,
    Control, and Optimization), funded by the Interuniversity
    Attraction Poles Programme, initiated by the  Belgian
    Federal Science Policy Office, and by the Concerted Research Action (ARC) of the
    French Community of Belgium.}
}
\date{\today}

%MSC
%60J10 (Primary)
%05C80 (Secondary)

\maketitle

\begin{abstract}
We investigate the mixing rate of a Markov chain where a combination of \emph{long
  distance edges} and \emph{non-reversibility} is introduced: as a
first step, we focus here on the following graphs: starting from the cycle
graph, we select random nodes and add all edges connecting them. We
prove a square factor improvement of the mixing rate compared to the
reversible version of the Markov chain.

\bigskip

\end{abstract}

\section{Introduction}

We study the mixing properties of certain Markov
chains which describe how fast the distribution of
the state approaches the stationary distribution regardless of the
initial conditions. The overall goal is to provide significant
improvement of mixing by minor modifications of the Markov chain.

Mixing time and rate are fundamental quantities in the study of Markov
chains \cite{levin:2009markov}, \cite{montenegro_tetali:slowmixing2006} and are the object of
active research; they are also highly relevant to applications
where mixing properties are strongly tied with performance
metrics. This is for example the case for Markov chain Monte Carlo
methods, which provide cheap approximations for integrals, and also
allow sampling from complex distributions that would otherwise be hard
to generate directly \cite{jerrum:mcmc1998}. Markov chains also provide a powerful scheme
for approximating the volumes of high dimensional convex bodies
\cite{lovasz2006hit}, \cite{lovasz2006simulated}. A different
application, average consensus, involves the
distributed computation of the average of initial values at different
agents in a multi-agent system (values which might correspond to
measurements, opinions, etc.) \cite{olsh_tsits:consensus_speed2009}. This can be achieved using a
Markov chain for which the stationary distribution is uniform. The
initial values can be viewed as a probability distribution scaled by a
constant, and the Markov chain will approach the uniform distribution
multiplied by the same constant, therefore the average of the initial
values will be present at each node. Efficient consensus (and average
consensus) approaches actually also play a role in several recent
distributed optimization algorithms \cite{nedic_olsh:optim_varigraph2015},
\cite{nedic_asu_pparillo:cons_optim}.
For all these applications good
performance is crucial and it is
determined by the dynamics of the underlying Markov chain. Mixing
properties are formulated exactly to answer such questions, which are
the topic of the current paper.

There are several ways of obtaining good mixing performance. In many
applications, the graph of possible transitions is determined by the
problem definition, but the specific transition probabilities can be
chosen. When these are required to satisfy some strong symmetry
properties (reversibility, described later in detail)
choosing those to optimize the mixing rate can be formulated as an
SDP problem \cite{boyd_and_al:fastmix2009},
\cite{boyd_and_al:fastmix2004}, which can be solved numerically using
standard methods. Departing from these symmetry properties brings
strong technical challenges, at the same time it can actually lead to significant improvement; the
mixing time can indeed drop to its square root in some cases
\cite{diaconis:nonrev2000}, \cite{montenegro_tetali:slowmixing2006}.
On the other hand, when it is possible to modify the graph of
possible transitions, astonishing speedup can also be obtained by
adding even a small number of randomly selected edges
\cite{addarioberry:swnmixing2012}, \cite{durrett:rgd},
\cite{krivelevich2013smoothed}.

Our long term goal is to study the speedup that can be achieved by a
combination of two a priori orthogonal transformations: (i) the
addition of a small number of random edges, and (ii) the introduction
of a strong non-reversibility.
We start with a cycle graph of $n$ nodes, select a lower number
$k$ of them to become hubs, then add extra edges between
the hubs. This scheme is motivated by one of the renown models to represent Small World
Networks, the Newman-Watts model \cite{durrett:rgd},
\cite{newman:swn2000}. The cycle presents a
natural way of including asymmetry by introducing a \emph{drift}
meaning increased clockwise transition probabilities along the cycle
and decreased counter-clockwise ones.
At this stage the model needs three parameters to be specified: the
placement of the hubs, the added interconnection structure on them,
and the asymmetry introduced along the cycle.

In this paper we consider as a first step a model where hubs are chosen randomly,
\emph{all edges between hubs are included} and asymmetry is taken to the
extreme: the Markov chain is a pure drift along the cycle taking
deterministic clockwise steps, except at the hubs. 
To better understand the dynamics of the process, observe that the
state of the Markov chain can be described by an arc (as the cycle
is split by the hubs) and the position within that arc. The main challenge
here is to show mixing happens both in term of arcs and in terms of
positions within. We prove that this model reaches a mixing rate of $\Omega(k/n)$ (up to $\log n$ factors) if $k=n^\sigma$.

In comparison, if we were to put pure drift along the cycle but with equidistant hubs, we
would have rapid mixing in term of the arcs (a perfect one after leaving
the first arc), but no mixing at all in term of the position on the
arc. Even by decreasing the drift or changing the interconnection
structure, the mixing rate will remain $O((k/n)^2)$ \cite{gb:ringmixing2014}.

Furthermore, if we were to stay with the classical, symmetrical
transitions along the cycle, the mixing rate will be again
$O((k/n)^2)$: for an arc at least $n/k$ long even the hitting time of
the ends of the arc from the middle is $\Omega((n/k)^2)$. 
This holds for any hub placement and interconnection structure.

After all, we want to emphasize that a speedup with a mixing rate of
$\Omega(k/n)$ is feasible only now that both random hubs and a drift
along the cycle are implemented, as opposed to only one of these.

The rest of the paper is organized as follows. In Section
\ref{sec:preliminaries} we formally describe the random graph model
and Markov chain on which we focus and a proxy graph model that we will use for the analysis.
In Section \ref{sec:randpol} we prove
the main mixing rate result for this proxy graph model. We then
translate our result to the primary graph model in Section
\ref{sec:mainmodel}. In Section \ref{sec:simulations} simulations are
presented complementing our asymptotic analytical results. We also
demonstrate how the mixing rate changes when the drift is decreased
for the model, suggesting that further performance improvements might
be possible. We draw conclusions and outline possible future
research directions in Section \ref{sec:conclusions}.
 
\section{Graph models, Markov chains and mixing rates}
\label{sec:preliminaries}

The concept of the graphs we consider is the following. We start with
a cycle with $n$ nodes, and randomly select a low number of vertices, $n^{\sigma}$ out
of the total of $n$ for some $0<\sigma<1$, which become hubs. Then we connect all hub nodes with
each other. Let us now present the precise definitions.

\begin{definition}
  \label{def:graphmodel_fixn}
  Given $n,k\in\mZ^+$ we define the random graph distribution $B_n(k)$
  as follows. Starting from a cycle graph on $n$ nodes, we randomly
  uniformly choose among the $k$ element subsets of edges and we
  remove the edges in the chosen subset.
  For the $i$th remaining arcs, $1\le i \le k$, we mark the clockwise
  endpoint as $a_i$ and the other end as $b_i$. Then we add all edges
  $(b_i,a_j)$, for all $1\le
  i,j\le k$. An example is given in Figure \ref{fig:randgraphexample}.
\end{definition}

\begin{figure}[h]
  \centering
  \includegraphics[width=0.4\textwidth]{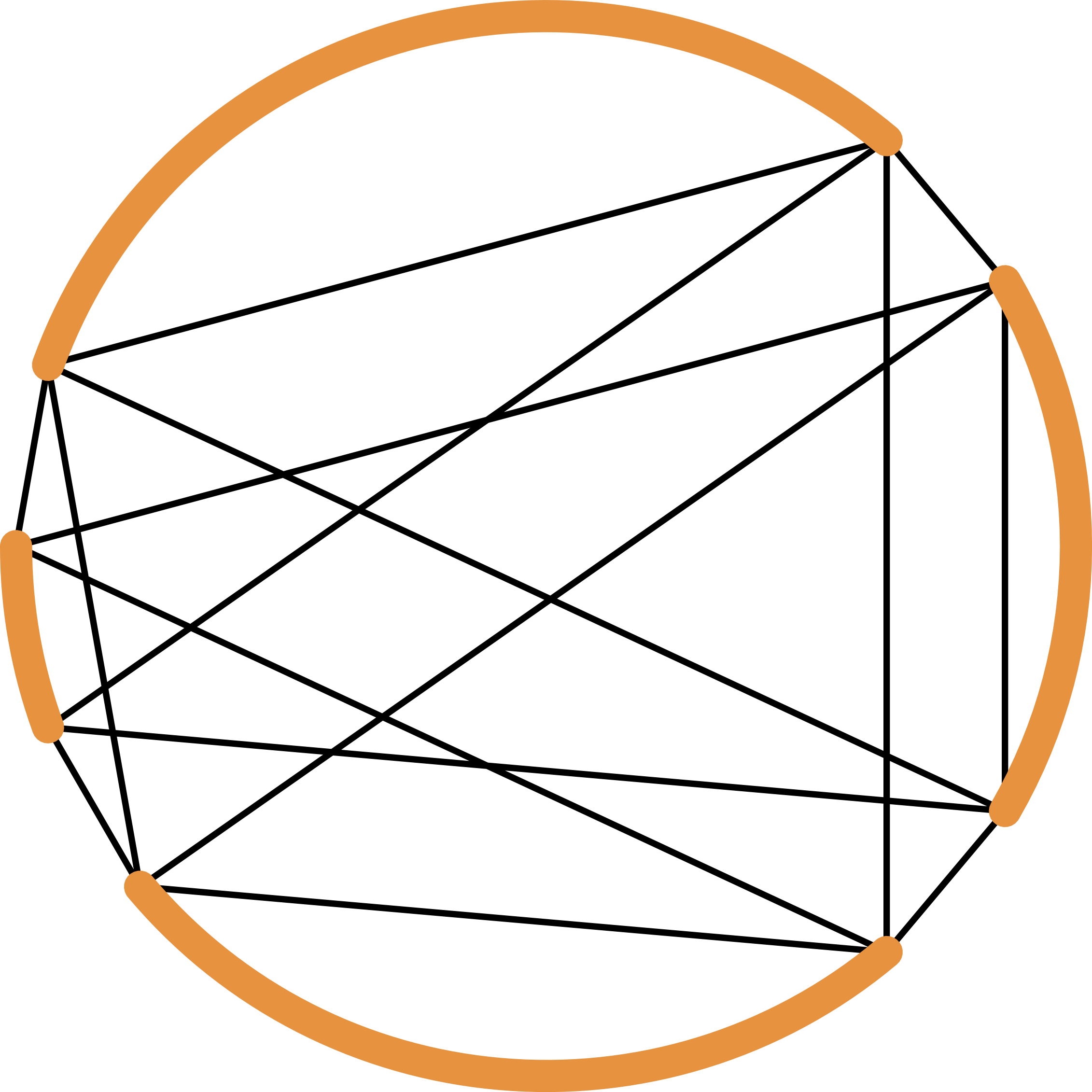}
  \caption{Example graph from $B_n(k)$ (see Definition \ref{def:graphmodel_fixn}).}
  \label{fig:randgraphexample}
\end{figure}

We are interested in the mixing behavior of Markov chains on these
graphs. A Markov chain is \emph{reversible} if for any edge $(u,v)$ of
the graph the probability of the $u\ra v$ transition is the same as the
$v\ra u$ transition. In this paper we go further from the comfortable
domain of reversible Markov chains, let us now introduce the ones
we will focus on.
\begin{definition}
  \label{def:puredriftMC}
  For any graph coming from $B_n(k)$ 
  %or $B(L,k)$ 
  we define the \emph{pure drift
    Markov chain} as follows. Within any arc we set transition
  probabilities to 1 along the arc all the way from $a_i$ to
  $b_i$. From any $b_i$, we set transition probabilities to $1/k$ on all edges
  towards all $a_j$. A part of such a chain is visualized in Figure
  \ref{fig:puredriftMCexample}.
\end{definition}

\begin{figure}[h]
  \centering
  \includegraphics[width=0.3\textwidth]{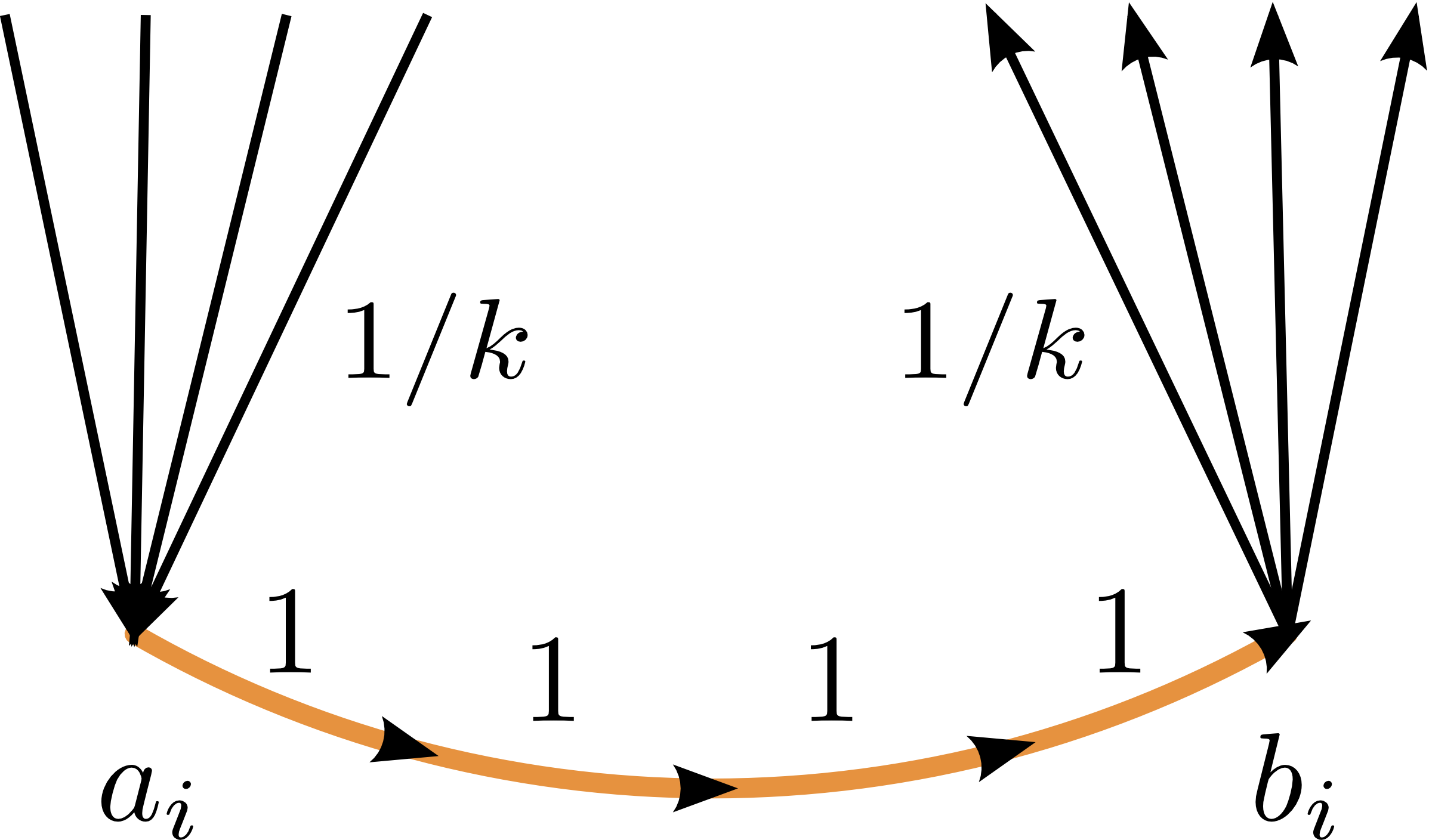}
  \caption{Example arc from the pure drift Markov chain (see
    Definition \ref{def:puredriftMC}).}
  \label{fig:puredriftMCexample}
\end{figure}

This is a Markov chain which has a doubly stochastic transition matrix,
therefore the stationary distribution is uniform. We
want to analyze the asymptotic rate as the distribution
approaches the stationary distribution. This mixing performance of the
Markov chain will be measured by the mixing rate:

\begin{definition}
  \label{def:mixingrate}
  For a Markov chain with transition matrix $P$ we define the
  \emph{mixing rate} $\lambda$ as
  $$\lambda = \min\left\{1-|\mu|~:~\mu\neq 1 \text{ is an eigenvalue of }
    P\right\}.$$
\end{definition}

Observe that for large numbers of nodes and comparably low number
of hubs, arc lengths approximately follow a geometric
distribution. However, they are not independently distributed, and
this approximation is not valid for large lengths.
Hence we introduce an alternative, simpler model of graphs reflecting approximately the same
concept but technically more convenient due to the added
independence. We will first establish our result for this alternative
model, and then extend it in Section \ref{sec:mainmodel} to the
original model of Definition \ref{def:graphmodel_fixn}.

\begin{definition}
  \label{def:graphmodel_indeparc}
  Given $L\in (1,\infty)$, $k\in\mZ^+$, we define the random graph
  distribution $B(L,k)$. Let us take independently $k$ random variables
  $$L_i \sim Geo(1/L),\qquad i=1,\ldots,k,$$
  where $Geo(p)$ denotes a geometric random distribution with
  parameter $p$ (and 1 as the smallest possible value).
  We begin with a graph that is the disjoint union of $k$ \emph{arcs},
  paths with $L_i$ nodes, each of which has a ``start point'' $a_i$ and an
  ``end point'' $b_i$, with $a_i=b_i$ if $L_i=1$. Then we add all edges $(b_i,a_j)$, for $1\le
  i,j\le k$.
\end{definition}

The extension of the pure drift Markov chain (Definition \ref{def:puredriftMC}) to this graph model is immediate.
Note that we chose to have $L_i$ denote the number of nodes as opposed
to the length of the path, because it leads to
simpler expressions in the technical developments.

\section{Random polynomials for pure drift Markov chains}
\label{sec:randpol}

In order to find the mixing rate of a Markov chain, we have to know
the eigenvalues of its transition matrix. For the current case we
transform this eigenvalue problem into finding the roots of a certain
polynomial.

\begin{proposition}
  \label{prp:polyrootequiv}
  Let us consider the pure drift Markov chain on a random graph from
  $B(L,k)$. Define also
  \begin{equation}
    \label{eq:qdef}
    q(z) = \sum_{i=1}^k z^{-L_i}.
  \end{equation}
  Then for $\mu\neq 0$, $\mu\in\mC$ is an eigenvalue of the transition matrix $P$ if and only if
  $q(\mu) = k.$
\end{proposition}
\begin{proof}
  Assuming $\mu$ is the eigenvalue of the transition matrix $P$ let us
  find the corresponding eigenvector $x$. Observe that each
  $a_i$ has incoming edges from exactly the same nodes and the same
  weights, so the eigenvector must take the same value $x_{a_i}=h$ at
  each of them for some $h$.

  For two subsequent nodes $p$ and $p^+$ along an arc, from $xP=\mu x$ we get
  $$x_p = \mu x_{p^+}.$$
  This implies that along the arcs we see the values
  $h,h\mu^{-1},\ldots,h\mu^{-L_i+1}$ (and thus just $h$ in case of a single node arc).
   This already completely
  determines $x$ up to scaling and ensures the eigenvalue equation for all nodes
  except the $a_i$. We get a valid eigenvector if the equation also
  holds for $a_i$, which takes the form
  $$\sum_{i=1}^k \frac{1}{k} h \mu^{-L_i+1} = \mu h.$$
  Having a non-zero eigenvector implies $h\neq 0$. Therefore the
  above equation is equivalent to $q(\mu) = k$ if $\mu\neq 0$.

  For the other direction, given a
  $\mu\neq 0$ such that $q(\mu) = k$, we can again build $x$ by setting $1,
  \mu^{-1},\ldots,\mu^{-L_i+1}$ on each arc, and this
  will clearly be an eigenvector of $P$ with
  eigenvalue $\mu$.
\end{proof}

By Definition \ref{def:mixingrate}, the mixing
rate is high if the transition matrix $P$ has no eigenvalue near
the complex unit circle. Therefore to get a lower bound on the mixing
rate we have to exclude a ring shaped domain for the eigenvalues.
The region to be avoided is
\begin{equation}
  \label{eq:Rgammadef}
  R_\gamma = \left\{z: 1 - \frac{1}{L\log^\gamma k} \le |z| \le 1,~
  z\neq 1 \right\},  
\end{equation}
where $\gamma$ is a constant parameter to be chosen later. We will show that
asymptotically almost surely (a.a.s.) no eigenvalue of $P$ falls in $R_\gamma$.
The width of the ring should be viewed as follows. We assume $L$ and $k$ are of similar magnitudes meaning that they
have a polynomial growth rate w.r.t.\ each other. Therefore the width
is at most a logarithmic factor lower than $1/L$.
Our key result is the following:

\begin{theorem}
  \label{thm:noroot}
  Assume $k,L \ra \infty$ while $\rho_l < \log L/\log k < \rho_u$ for some
  constants $0<\rho_l<\rho_u<\infty$, and fix $\gamma > 4$. We use the graph
  model $B(L,k)$ and the definition of $q(z)$ from \eqref{eq:qdef} and $R_\gamma$
  from \eqref{eq:Rgammadef}. Then for any $c,d>0$ we
  have
  $$P \left(\exists z \in R_\gamma, ~q(z)=k \right)  =
  O(k^{-c}L^{-d}).$$
  Consequently, in view of Proposition \ref{prp:polyrootequiv} and the definition of $R_\gamma$, we obtain for the mixing rate $\lambda > 1/(L\log^\gamma k)$ a.a.s.
\end{theorem}

We show the claim in four steps. First we ensure that we can assume
the $L_i$ variables to be bounded with high probability, this will
make further estimates possible.
We then check $z$ coming from
different parts of $R_\gamma$: we start with positive reals,
then we treat complex numbers in two different ways depending on their arguments.

Intuitively the reason is the following. For some real
$0 < z < 1$, $q(z)$ will be too large. Next, take $z$ with low
arguments, now all $z^{-L_i}$ will be in the same half-plane
resulting in a non-zero imaginary part for $q(z)$. When the argument is far enough
such that $q(z)$ has a chance to have zero imaginary part again, the
$z^{-L_i}$ will point in various different directions so that the
cancellations will force the real part below $k$, a.a.s.

Now let us make all this precise. We will confirm that each of the
intuitive steps above work with high probability, with the fourth requiring the $L_i$
to be different enough and not being extremely large. We then join these steps to give a proof of the theorem.

The probabilistic upper bound we need on the $L_i$ can be formulated in the
following way:

  \begin{lemma}
    \label{lm:lupperbound}
    For any $C>1,~L\ge 2$, there holds
    $$
    P(\max_i L_i \ge CL \log k) = O\left(k^{1-C}\right).
    $$
  \end{lemma}
  \begin{proof}

    Assume $k\ge 2$. Remember that $L_i$ are i.i.d.\ variables with law
    $Geo(1/L)$. Therefore we have that
    $$P(L_i \ge CL \log k) = P(L_i \ge \lceil CL \log k \rceil)
    = \left(1-\frac{1}{L}\right)^{\lceil CL \log k\rceil-1}
    \le 2\left(1-\frac{1}{L}\right)^{CL \log k},$$  
    based on $\left(1-\frac{1}{L}\right)^{\lceil CL \log k\rceil}
      \le \left(1-\frac{1}{L}\right)^{CL \log k}$ and
      $\left(1-\frac{1}{L}\right)^{-1}<2$.
    Knowing $(1-1/L)^L < 1/e$ we get
    $$P(L_i \ge CL \log k) \le 2e^{-C\log k} = 2k^{-C}.$$
    To treat all $L_i$ together for $1\le i \le k$, we us a simple union bound.
    $$P(\max_i L_i \ge CL \log k) \le 2k^{1-C}.$$
  \end{proof}
  
  From now on, we will only investigate the 
  (a.a.s.) event that the
  maximal $L_i$ is small as shown in Lemma \ref{lm:lupperbound}. Let
  us call this event $S(C)$. We now check $z$ coming from different parts of
  $R_\gamma$. The simplest case is when $z$ is a positive real:

  \begin{lemma}
    \label{lm:zreal}
    Assume $z\in (0,1)$. Then $q(z)\neq k$.
  \end{lemma}
  \begin{proof}
    For such $z$, 
    $q(z)$ is composed of $k$ positive terms $z^{-L_i}$, each of
    them being larger than 1 because all $L_i$ are positive (Remember
    that the smallest possible value for the numbers of nodes $L_i$ is
    1). Consequently the sum of the $z^{-L_i}$ is higher than $k$. 
  \end{proof}

  Next we show that there is no $z\in R_\gamma$ with small arguments
  for which $q(z)=k$.

  \begin{lemma}
    \label{lm:qimag}
    Assume $S(C)$ and take $z\in R_\gamma$ such that $0 \neq |\arg(z)| <
    \pi/(CL \log k)$. Then $\Im q(z) \neq 0$. Consequently $q(z)\neq k$.
  \end{lemma}
  \begin{proof}
    Without loss of generality, assume $\arg(z)>0$. The event
    $S(C)$ ensures $0< L_i < CL \log k$, so the $z^{-L_i}$ will all be in the
    same half-plane, as they will have
    an argument in $(-\pi, 0)$. For all these values, the imaginary
    part is negative, and the same holds thus true for their sum. Therefore it simply cannot be 0.
  \end{proof}

  It remains to check the elements of $R_\gamma$ whose argument is
  ``large''. 
  The arguments in question are those in
  \begin{equation}
    \label{eq:Adef}
    A := \left[ \frac{\pi}{CL \log k}, 2\pi - \frac{\pi}{CL \log
        k} \right].
  \end{equation}
  We argue that the arguments of $z^{-L_i}$ become so different that
  strong cancellations will happen. 
  We now formalize this idea in terms of the sum of cosines of the
  arguments $L_ix$. Note that the proposition statement uses
  $\cos^+(y) =\max(\cos(y),0)$ instead of simple cosines because we
  will need to take the sums of $\cos^+$
  scaled by different magnitudes $\left|z^{-L_i}\right|$ in \eqref{eq:req_estimate1}.

  \begin{proposition}
    \label{prp:cosloss}
    Choose constants $\alpha,\beta > 1$ and also $\rho_l,\rho_u$ as
    in Theorem \ref{thm:noroot}, and require
    $\rho_l < \log L/\log k < \rho_u$, where  $k,L,L_i$, defined in Definition
    \ref{def:graphmodel_indeparc}, are the number of arcs, the expected arc length, and the actual lengths of the arcs.
    We define
    $$m = \lceil \log^\alpha k \rceil, \qquad \delta = \log^{-\beta} k.$$
    Then for $k,L$ large enough we have, 
        $$P\left(\sup_{x\in A}\sum_{i=1}^m \cos^+(L_i  x) < m - \delta^2\right)
    \ge \frac{1}{3},$$
    where $\cos^+(y) = \max(\cos(y),0).$
  \end{proposition}

  \begin{proof}
  
    Broadly speaking we will show that at least one of the $L_i x$ terms will be far
    from $2k\pi$, which should decrease enough the sum of the
    cosines. For a single $x$, we state the following lemma.
  \begin{lemma}
    \label{lm:indprob}
    Use $k,L,L_i$ as in Definition \ref{def:graphmodel_indeparc}.
    Fix $\beta >1$, $x\in A$ (as defined in \eqref{eq:Adef}) and
    choose an arbitrary
      modulo $2\pi$ interval $D\subset [0,2\pi]$ of size
      $|D|=6\delta=6\log^{-\beta} k$. 
    Then for $k,L$ large enough we have
    $$P(\{L_1 x\} \in D) \le \frac{2}{3},$$
    where $\{a\}$ stands for $a \mod 2\pi$.
  \end{lemma}
  \begin{proof}
    Each element of the series $\{x\},\{2x\},\{3x\},\ldots$ is
    either in $D$ or not. We can therefore split the series into
    blocks that are in $D$ and to blocks that are not. Let
    $t_1$ be the first coefficient such that $\{t_1 x\}\in D$, then we
    can define the blocks
    \begin{align*}
      &\{t_i x\}, \{(t_i+1) x \}, \ldots, \{ (s_i-1)x \} \in D,\\
      &\{s_i x\}, \{(s_i+1) x \}, \ldots, \{ (t_{i+1}-1)x \} \notin D,
    \end{align*}
    with $s_i\geq t_i+1$ and $t_{i+1}\geq s_i+1$. Observe that 
    \begin{align} \label{eq:det_proba}
      P(\{L_1 x\} \in D) &= \sum_{i\geq 1} P(L_1 \in [t_i, s_i-1]) \\
                         &=1 - \sum_{i\geq 1} P(L_1\in  [s_i, t_{i+1}-1]) - P(L_1 < t_1) \nonumber \\
                         &\leq 1- \sum_{i\geq 1} P(L_1\in  [s_i, t_{i+1}-1]). \nonumber 
    \end{align}
    We will now show 
    \begin{equation}
      \label{eq:intervalcomp}
      P(L_1 \in [t_i, s_i-1]) \le 2 P(L_1 \in [s_i, t_{i+1}-1]),
    \end{equation}
    which together with \eqref{eq:det_proba} will allow to conclude.

    For this purpose, we first compare the number of elements in the blocks
    above by relating $t_{i+1}-s_i$ with $s_i-t_i$. We claim for $k,L$
    large enough that
    \begin{equation}
      \label{eq:stcomparison1}
      t_{i+i}-s_i \ge s_i-t_i.
    \end{equation}
    Indeed, for such $k$ we have $|D|=6\delta < \pi/2$. Without the loss of
    generality we assume $x \in [0,\pi]$.
    
    If $s_i-t_i=1$, then we immediately get $t_{i+i}-s_i \geq
      s_i-t_i$. Otherwise, there are at least two consecutive elements
      of the series in $D$ and the length $6\delta$ of $D$ is thus at least $(s_i-t_i-1)x$. Therefore we have
      \begin{equation}\label{eq:boundstin}
      x \le \frac{6\delta}{s_i-t_i-1}\le \frac{\pi}{2(s_i-t_i-1)}.
      \end{equation}
      A simple consequence
      is $x\le 6\delta \le \pi/2$. Also, by rearranging we get
      \begin{equation}
        \label{eq:st_bound1}
        x(s_i-t_i) \le \frac{\pi}{2} + x.
      \end{equation}
      For the elements outside $D$, observe that since $(s_i - 1)x$ is in $D$ and $t_{i+1}x$ is again in $D$, the $t_{i+1} - s_i$ intervals of length $x$ defined by $[(s_i-1) x, s_i x], [s_i x, (s_i +1)x ], \dots, [(t_{i+1}-1)x , t_{i+1}x]$ must cover at least the length of the complement of $D$, i.e. at least $2\pi - 6\delta$. We have thus
      $$t_{i+1}-s_i \geq \frac{2\pi-6\delta}{x}-1 \geq
      \frac{\frac{3}{2}\pi}{x}-1.$$
      Rearranging yields
      $$x(t_{i+1}-s_i) \ge \frac{3\pi}{2} - x.$$
      We compare this with \eqref{eq:st_bound1}, note that $0<x\le
      \pi/2$ and conclude again that $t_{i+i}-s_i \geq s_i-t_i$.

      We now show an upper bound on the block sizes. Remember
      that $x\in A$ defined in \eqref{eq:Adef} implies $x\geq
      \frac{\pi}{CL \log k}$. So using \eqref{eq:boundstin} and $\delta = \log^\beta k$, we obtain
    \begin{equation}
      \label{eq:stcomparison2}
      s_i-t_i \le \frac{6\delta}{x} + 1 \le \frac{6\delta}{\pi/CL\log k} + 1=
      \frac{6}{\log^\beta k} \cdot \frac{CL\log k}{\pi} + 1 \le  2CL
      \log^{1-\beta}k + 1.
    \end{equation}

    Let us come back to the probabilities of $L_1$ falling within the blocks defined.
    As $L_1$ is a geometric random variable, shifting the interval of
    interest by $s_i-t_i$ introduces only a simple multiplicative factor:
    $$ P(L_1\in [t_i,s_i-1]) = \left(1-\frac{1}{L}\right)^{t_i-s_i} P\left( L_1 \in \left[s_i,
        s_i+(s_i-t_i)-1\right]\right) \le \ldots $$
    We enlarge the target interval from length $s_i-t_i$ to
    $t_{i+1}-s_i$ relying on
    \eqref{eq:stcomparison1}. Clearly by this the probability cannot decrease.
    \begin{equation}
      \label{eq:blockineq1}
      \ldots \le \left(1-\frac{1}{L}\right)^{t_i-s_i} P\left( L_1 \in \left[s_i,
          t_{i+1}-1\right]\right).
    \end{equation}

  For the coefficient at the end of
  \eqref{eq:blockineq1} we use \eqref{eq:stcomparison2} to get for
  $k, L$ large enough
  $$\left(1-\frac{1}{L}\right)^{t_i-s_i} \le
  \left(1-\frac{1}{L}\right)^{-2CL\log^{1-\beta}k-1} \le
  \exp(3C\log^{1-\beta}k)\left(1+\frac{2}{L}\right)\le 2.$$
  During these estimates we used $(1-1/L)^{-2L} \le \exp(3)$ and $(1-\frac{1}{L})^{-1}\leq 1+ \frac{2}{L}$
  for $L$ large enough, and $\log^{1-\beta} k$ being as close to
  0 as needed for $k$ large enough.
  
  Substituting this last bound together into \eqref{eq:blockineq1}
  leads to \eqref{eq:intervalcomp}. This is enough to complete the proof as
  we have seen before.
\end{proof}

  To come back to the proof of Proposition \ref{prp:cosloss}, let us
  choose $D = [-3\delta, 3\delta]$. For small enough $\delta$,
  whenever we have $\{L_ix\}\notin D$, it implies $\cos^+(L_i x) <
  1 - 2\delta^2$. Assuming $\delta$ to be small enough is again
  equivalent to another (independent) bound for $k$ to be large enough. Using Lemma  \ref{lm:indprob} and knowing that the $L_i$
  are independent random variables we have
  \begin{equation}
    \label{eq:cossumsinglex}
    P\left(\sum_{i=1}^m \cos^+(L_i x) \ge m - 2\delta^2 \right)
    \le P\left(\{L_ix\}\in D\right)^m
    \le \left(\frac{2}{3}\right)^m,
  \end{equation}
  where $m = \log^\alpha k$ was defined in the statement of the proposition.
  This is the type of probability bound we are looking for, but only for a
  single $x$. Next we extend it to
  all $x\in A$ simultaneously, where $A$ is the interval of interest
  of arguments \eqref{eq:Adef}. As an intermediate step, break $A$
  into $\lceil A/\epsilon\rceil$ equal subintervals with $\epsilon =
  2\delta^2/(m CL \log k)$ and choose $x_j$ as the middle of each of these
  subintervals. In this setting, no point of $A$ is further than
  $\epsilon/2$ from some point $x_j$. Using the union bound for these
  chosen points we see
  \begin{equation*}
      P\left(\sup_j\sum_{i=1}^m \cos^+(L_i
        x_j) \ge m - 2\delta^2 \right) 
      \le \sum_{j} P\left(\sum_{i=1}^m \cos^+(L_i
        x_j) \ge m - 2\delta^2 \right)
\end{equation*}        
Taking into account the number of points $x_j$ in the grid and \eqref{eq:cossumsinglex}, we obtain
\begin{equation}
    \begin{split}\label{eq:probboundex}
      P\left(\sup_j\sum_{i=1}^m \cos^+(L_i x_j) \ge m - 2\delta^2 \right) 
      &\le \left\lceil \frac{2\pi m CL \log
      k}{2\delta^2}\right\rceil \left(\frac{2}{3}\right)^m\\
      \le \pi(\log^\alpha k+1) C L \log k \log^{2\beta}k
      \left(\frac{2}{3}\right)^{\log^\alpha k}
      &= C\pi L\log^{2\beta+\alpha+1} k \left(\frac{2}{3}\right)^{\log^\alpha k}(1+o(1)).
      \end{split}
    \end{equation}
  In this final term $(2/3)^{\log^\alpha k}$ decreases faster than
    the inverse of any polynomial in
  $k$ as $\alpha>1$. All others
  parts of the product have polynomial or lower rate in $k$.
  Consequently we see that the right hand
  side of \eqref{eq:probboundex} will become arbitrarily small as
  $k,L$ grows. In particular, it will go below $2/3$.

  At this point we have the probability estimate for grid points $x_j$. We
  need to extend this to the complete interval $A$, introduced in \eqref{eq:Adef}.
  We show that for $k,L$ large
  enough we have
  \begin{equation}
    \label{eq:gridextension}
    P\left(\sup_{x\in A} \sum_{i=1}^m \cos^+(L_i x) \ge m -
      \delta^2 \right) < P\left(\sup_j\sum_{i=1}^m \cos^+(L_i
      x_j) \ge m - 2\delta^2 \right).
  \end{equation}
  Indeed, for any $x\in A$ there is a grid point $x_j$ at most
  $\epsilon/2$ away. As the derivative of $\cos^+$ stays within
  $[-1,1]$, the change of the sum when moving to $x_j$ from $x$ is at
  most
  $$\sum_{i=1}^m \frac{\epsilon}{2} L_i \le m\frac{\delta^2}{m
    CL \log k} CL \log k = \delta^2.$$
  Therefore when the sum on the left hand side
  of \eqref{eq:gridextension} is at least $m-\delta^2$ for a certain
  $x\in A$ then there
  also must be a grid point $x_j$ for which the sum is at least
  $m-2\delta^2$. The inclusion of the events shows the inequality for
  the probabilities.
  \end{proof}

  \begin{proof}[Proof of Theorem \ref{thm:noroot}]
    Note that the right hand side of the claim can be slightly simplified. The relation
      between $L$ and $k$ ensures $k^{\rho_l} < L <
      k^{\rho_u}$ after a while, therefore the $L$ term on the right
      hand side can be replaced by a power of $k$. It is now sufficient to show
      that the probability in question is $O(k^{-c})$ for all $c>0$,
      we will thus consider only this case.

  Choose $C \ge c+1$. Let us assume $\max_i L_i < CL\log k$. This not
  being true is an exceptional event of probability $O(k^{-c})$ as shown
  in Lemma \ref{lm:lupperbound}.
  In order to exclude the roots from all $R_\gamma$, we split this
  region into three parts, and show that $z$ cannot be a solution of $q(z) = k$ in any of these parts.

  When $0<z<1$ is a positive real, it cannot be a solution according to
  Lemma \ref{lm:zreal}. When $z$ has a small argument, that is, $|\arg(z)| < \pi/(CL\log
  k)$, we refer to Lemma \ref{lm:qimag} to confirm $z$ cannot be a solution either.

  The remaining case is when $z$ has a large argument, that is,
  $\arg(z) \in A$.
  We aim to bound $\Re q(z)$. On one hand, we estimate
  the magnitude of the terms $z^{-L_i}$. Then we combine these with
  the cosines of
  the arguments to find their contributions to the real part of
  $q(z)$. Here we rely on Proposition \ref{prp:cosloss}, but let us
  make this precise.

  We need to check the magnitude of the terms $z^{-L_i}$. Knowing $|z| > 1-1/(L\log^\gamma k)$ and $L_i
  < CL\log k$ for $k,L$ large enough we have
  \begin{align*}
    |z|^{-L_i} &\le \left(1-\frac{1}{L\log^\gamma k} \right)^{-CL\log
      k} \le \left(1+\frac{2}{L\log^\gamma k} \right)^{CL\log
      k}\\
    &\le \exp(2C\log^{1-\gamma}k) \le 1 + \frac{4C}{\log^{\gamma-1}k}.
  \end{align*}
  Considering $\Re q(z)$, this gives
  \begin{equation}
    \label{eq:req_estimate1}
    \Re q(z) = \sum_{i=1}^k |z|^{-L_i}\cos(L_ix) \le \sum_{i=1}^k
    |z|^{-L_i}\cos^+(L_ix) \le
    \left(1+\frac{4C}{\log^{\gamma
          -1}k}\right) \sum_{i=1}^k \cos^+(L_ix).
  \end{equation}
  Note that the last inequality is the reason we have been working with $\cos^+$, as it would generally not hold true for $\cos$.
  Let us arrange the $k$ elements of this sum into groups of $m=\log^\alpha
  k$ arbitrary elements, consequently resulting in $k/m$ such
  groups. For a moment we assume $k$ is divisible by $m$.
  Let the sum of these groups be $S_1, S_2, \ldots, S_{k/m}$. According to
  Proposition \ref{prp:cosloss} we have
  $$P(S_i \ge m - \delta^2) \le \frac{2}{3},$$
  and each of these events are independent. Therefore the number of
  such events happening follows a $Binom(k/m,r)$ distribution for some
  $r\le 2/3$. From standard Chernoff bounds we see that
  $$P\left(Binom\left(\frac{k}{m},r\right) > \frac{3k}{4m}\right) \le 
  P\left(Binom\left(\frac{k}{m},\frac{2}{3}\right) >
    \frac{3k}{4m}\right) \le \exp\left(-\frac{k}{96m}\right) < k^{-c},$$
  for $k$ large enough. Consequently, at most $3k/(4m)$ of $S_i$ are
  larger than $m-\delta^2$. (with the exception of an event with probability of $O(k^{-c})$). In this case we have
  $$\sum_{i=1}^k \cos^+(L_i x) = \sum_{j=1}^{k/m} S_j \le \frac{3k}{4m}m
  + \frac{k}{4m}(m-\delta^2) = k - \frac{\delta^2k}{4m}.$$
  Plugging this back into \eqref{eq:req_estimate1} we arrive at
  \begin{equation}
    \label{eq:req_estimate2}
    \Re q(z) \le \left(1 + \frac{4C}{\log^{\gamma-1}k}\right) k
    \left(1 - \frac{\delta^2}{4m}\right) = k \left(1 + \frac{4C}{\log^{\gamma-1}k}\right)
    \left(1 - \frac{1}{4\log^{2\beta+\alpha}k}\right).
  \end{equation}
  Let us choose $\gamma > 2\beta + \alpha + 1$. With such a choice,
  the term $k$ above gets multiplied by a coefficient lower than 1
  for $k$ large enough. This shows $\Re q(z) < k$ which implies $z$ is
  not a solution of $q(z)=k$. This a.a.s.\ holds simultaneously for all $z\in R_\gamma,~\arg(z)\in A$.

  For the sake of completeness, if $k$ was not divisible by $m$,
    we could still form $\lfloor\frac{k}{m}\rfloor$ groups as before
    with $\bar{k} = m\lfloor\frac{k}{m}\rfloor$ elements, then collect
    the sum of the remaining $k-\bar{k}$ terms into
    $S_{\lfloor\frac{k}{m}\rfloor+1}$. Performing the same argument
    and using the trivial upper bound for $\cos^+$ when working with
    $S_{\lfloor\frac{k}{m}\rfloor+1}$ we get
    $$\Re q(z) \le \bar{k} \left(1 + \frac{4C}{\log^{\gamma-1}k}\right)
    \left(1 - \frac{1}{4\log^{2\beta+\alpha}k}\right) +
    (k-\bar{k})\left(1 + \frac{4C}{\log^{\gamma-1}k}\right).$$
    Knowing $k-\bar{k}<m$ the new additive term compared to
    \eqref{eq:req_estimate2} is (poly)logarithmic, which will not
    compensate for the almost linear $k/\log^\epsilon k$ type of loss
    originating from the first term. Hence once again, the right hand
    side will not reach $k$ and thus $z$ can not be a solution.

  Regarding the parameters, previously for Proposition \ref{prp:cosloss}
  we only needed to ensure $\alpha, \beta > 1$. Therefore we can apply
  Proposition \ref{prp:cosloss} and this reasoning for
  $\alpha=\beta=1+\epsilon,~\gamma=4+4\epsilon$ for any $\epsilon >
  0$, eventually allowing any $\gamma > 4$.
  
  Remember that during the proof, we had two small exceptional events, both having probability $O(k^{-c})$.
  This allows thus confirming the theorem with the condition on $\gamma$ and with the
  probability bound on the exceptional cases.

\end{proof}

Theorem \ref{thm:noroot} guarantees the absence of eigenvalues with
large absolute value (except at 1) with high probability. We can reformulate it in the following way
\begin{theorem}
  \label{cor:model1mixing}
  Assume $k,L \ra \infty$ while 
  $\rho_l < \log L/\log k < \rho_u$ for some constants
  $0<\rho_l<\rho_u<\infty$. Then for any $\gamma >4$ a.a.s.\ we have 
  the following bound on the mixing rate for $B(L,k)$:
  $$\lambda > \frac{1}{L \log^\gamma k}.$$
\end{theorem}
\begin{proof}
This lower bound is a direct consequence of Theorem \ref{thm:noroot} and the
definitions of $R_\gamma$ and the mixing rate $\lambda$. 
\end{proof}

\section{Mixing rates for $B_n(k)$}
\label{sec:mainmodel}

We now translate Theorem \ref{cor:model1mixing} to our initial graph model $B_n(k)$, where
the total number of nodes are fixed beforehand.  For this purpose, we first show that for any $L\in\mR^+$ the conditional distribution of
$B(L,k)$ conditioned on having $n$ nodes in total is the same as
$B_n(k)$ in the following sense. We will use the compact notation
$B(L,k)|_n$ for the aforementioned conditional distribution.

We need to map a graph from $B(L,k)|_n$ to the cycle. Given such a
graph $G$ (as we build it in Definition \ref{def:graphmodel_indeparc}) and
$s\in \{1,2,\ldots n\} = [n]$ define $T(G,s)$ as follows. Map
$a_1$ to node $s$, then progressing along the arc of $L_1$ to $s+1,s+2,\ldots$,
continuing with $L_2,L_3,\ldots$ taking the numbers $s+i$ modulo $n$ once
needed. Map the edges of $G$ consistently with the nodes, then $T(G,n)$ is a graph on the labeled nodes $[n]$.

\begin{proposition}
  \label{prp:model_equiv}
  Given are $n,k\in\mZ^+, ~L\in\mR^+$. Let $U(n)$ denote the uniform
  distribution on $[n]$ and $T^*(\cdot,\cdot)$ the
  induced measure transformation of $T(\cdot,\cdot)$ defined above.

  Then $T^*(B(L,k)|_n, U(n)) = B_n(k)$. Simply speaking, if we
  randomly choose the starting point where we map $B(L,k)|_n$ to the
  cycle, we get the distribution $B_n(k)$.
\end{proposition}

Note that the randomization in the mapping does not change the internal
structure of the graph, therefore once the proposition is proven, we
also immediately get the following:

\begin{corollary}\label{cor:cond_equal}
  Given $n,k\in\mZ^+, ~L\in\mR^+$ consider the conditional
  distribution $B(L,k)|_n$.
  Then the corresponding distribution of the mixing rate $\lambda$ is the same as
  the distribution of $\lambda$ for $B_n(k)$.
\end{corollary}

\begin{proof}[Proof of Proposition \ref{prp:model_equiv}]
  For any $l_1\ge 1$ we have $P(L_1=l_1) = (1-p)^{l_1-1}p$ and the arc
  of $a_1\cdots b_1$ will consist of $l_1$ nodes, (with $p=\frac{1}{L}$). Consequently, in
  the setup of $\{L_1=l_1,L_2=l_2,\ldots,L_k=l_k\}$ we get a total of
  $n$ nodes exactly if $\sum_{i=1}^k l_i = n$. The probability of
  such an instance is
  $$
  P(L_1=l_1,L_2=l_2,\ldots,L_k=l_k) = \prod_{i=1}^k(1-p)^{l_i-1}p
  = (1-p)^{\sum_{i=1}^k l_i}p^k = (1-p)^{n-k} p^k.
  $$
  This probability is independent of the choice of $\{L_i\}_{i=1}^k$,
  therefore the conditional distribution is uniform on all the
  possibilities.

  A uniform number on $[n]$ is supplemented and we need to
  relate this joint variable to the distribution of $B_n(k)$.
  It is straightforward to see that $T$ will map any
  element in the support of $B(L,k)|_n \times [n]$ to an element of
  the support of $B_n(k)$ as they are built the same way.

  Moreover, this will be a homogeneous map in the sense that each element of
  the support of $B_n(k)$ will be obtained exactly $k$ times. To see this, start
  from any such element, we will get the $k$ different preimages
  depending on which arc we choose to be $L_1$. (Note that two subsequent
  deleted edges might lead to an ``arc'' without edges, of size
    1 in our current notation.)

  In the end, both the conditional distribution $B(L,k)|_n$ and
  $U(n)$ were uniform, applying the map $T$ that is a uniform
  $k$-fold cover on a target space will result in a uniform
  distribution on its range, confirming the claim of the proposition.
\end{proof}

Our last ingredient is a lower bound on the probability for a random graph model $B(n/k,k)$ to have exactly $n$ nodes.

\begin{lemma}\label{lem:1/n}
The probability for a random graph $B(n/k,k)$ to have exactly $n$ nodes is at least $\frac{1}{n}$.
\end{lemma}
\begin{proof}
We have seen in the proof of Proposition \ref{prp:model_equiv} that all instances of $n$ nodes have the same probability
$\left(1-\frac{k}{n}\right)^{n-k}
\left(\frac{k}{n}\right)^k$. Moreover, standard combinatorial
arguments show that there are ${n-1 \choose k-1}$ ways of
  distributing the $n$ nodes into the $k$ arcs, leading to the
following  probability of obtaining exactly $n$ nodes, and event that
we denote by $M(n,k)$
\begin{equation}
    \label{eq:probM}
    P(M(n,k)) = 
{n-1 \choose k-1}\left( \frac{k}{n} \right)^k \left(1
      - \frac{k}{n}\right)^{n-k} = {n-1 \choose k-1}\frac{k^k
      (n-k)^{n-k}}{n^n}.  
  \end{equation}
  We develop a simple asymptotic estimate for this probability. From the
  Stirling formula we know that
  $$\lim_{n\ra \infty} \frac{n!}{\sqrt{2\pi
      n}\left(\frac{n}{e}\right)^n} = 1.$$
  For conciseness, we will use the $\approx$ relation if the ratio of
  the two quantities is 1 in the limit. In this spirit we get
  $${n \choose k} \approx \frac{\sqrt{2\pi
      n}\left(\frac{n}{e}\right)^n}
  {\sqrt{2\pi
      k}\left(\frac{k}{e}\right)^k
    \sqrt{2\pi
  (n-k)}\left(\frac{n-k}{e}\right)^{n-k}} = 
\sqrt{\frac{n}{2\pi k(n-k)}} \frac{n^n}{k^k (n-k)^{n-k}}.$$
Let us plug this back to \eqref{eq:probM}, while noting ${n-1 \choose
  k-1} = \frac{k}{n}{n \choose k}$.
  \begin{equation*}
    P(M(n,k)) \approx \frac{k}{n}\sqrt{\frac{n}{2\pi k(n-k)}} = \sqrt{\frac{k}{2\pi n(n-k)}}.
  \end{equation*}
  As a very crude bound we get for $n,k$ large enough that 
  \begin{equation}
    \label{eq:mbound}
    P(M(n,k)) > \frac{1}{n}.
  \end{equation}
\end{proof}

We can now extend Theorem \ref{cor:model1mixing} to our initial model $B_n(k)$.

\begin{theorem}
  \label{thm:model2mixing}
  Assume $n,k \ra \infty$ while $\rho_l < \log n/\log k < \rho_u$ for some
  constants $1<\rho_l<\rho_u<\infty$. Then for any $\gamma >4$ a.a.s.\ we have
  the following bound on the mixing rate for $B_n(k)$:
  $$\lambda > \frac{k}{n \log^\gamma k}.$$
\end{theorem}
\begin{proof}
We know from Theorem \ref{thm:noroot} that the probability of the polynomial $q(z)$ having a root in the forbidden ring (and thus of the mixing rate being smaller than $\frac{1}{L \log^\gamma k}$) bounded as $O(L^{-2} k^{-2}) =
  O(n^{-2})$. Hence it remains negligible with respect to the probability of $B(n/k,k)$ having exactly $n$ nodes (for $n,k$ large enough) which we have shown in Lemma \ref{lem:1/n} to be at least $1/n$. We deduce that the mixing rate of pure drift Markov chains for the conditional random graph model $B(n/k,k)|_{n}$ is also at least  
$\frac{1}{L \log^\gamma k}$ a.a.s. Proposition \ref{prp:model_equiv} allows then concluding that the same holds for $B_n(k)$, which concludes the proof.
\end{proof}

\section{Simulations}
\label{sec:simulations}

Following the asymptotic theoretical results we perform complementing simulations
to analyze the tightness of the bounds obtained. We also explore
further using numerical tools
for the next step of research that is not yet treated
analytically.

The mixing results we have are exciting as we see a strong speedup
compared to the similar reversible Markov chain 
with transition matrix $\tilde{P} = (P+P^\top)/2$.
By this we set all
transition probabilities on all edges to be equal in the two directions.
For this Markov chain, if the initial
distribution is concentrated in the middle of the longest arc, the
Central Limit Theorem ensures that even
after $\Omega(L^2\log^2 k)$ steps the probability of not leaving the
arc is bounded away from 0. Consequently we get a lower bound of the
same order for the mixing time and which in turn translates to the mixing rate bound
\begin{equation}
  \label{eq:reversiblerate}
  \lambda < \frac{C}{L^2\log^2 k},  
\end{equation}
which is a square factor worse than our new results for the
non-reversible Markov chain.

Simulations are in line with the speedup we see when comparing
\eqref{eq:reversiblerate} with Theorem \ref{cor:model1mixing}. Figure
\ref{fig:eig_comparison} is a log-log histogram showing the decrease of
$\lambda$ as the node count $n$ increases. The histogram presents the
simulation results for the non-reversible and reversible Markov chain
and we do observe the strong separation predicted by the theoretical
results. The stripe on the top presents $\lambda$ for the
non-reversible Markov chains while the bottom one corresponds to the
reversible ones.
\begin{figure}[h]
  \centering
  %\resizebox{0.6\textwidth}{!}{\input{Images/cycles_54_to_2980_alltoall_summary}}
  \resizebox{0.6\textwidth}{!}{\input{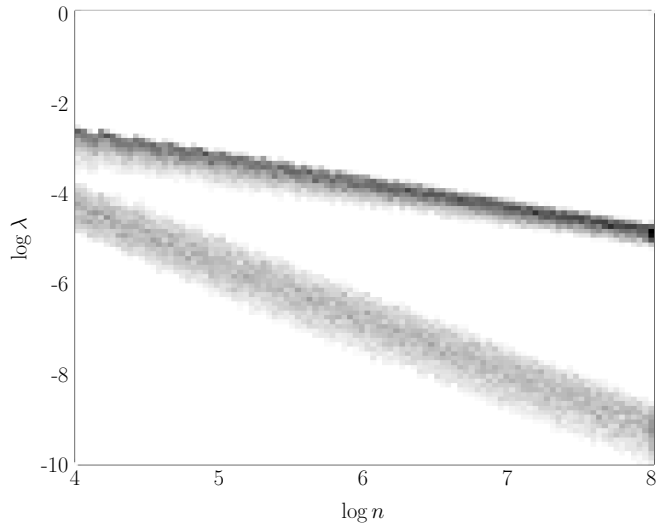}}
  \caption{Histograms for the mixing rates $\lambda$ for the Markov
    chains on the graphs
    $B_n(\lfloor \sqrt{n}\rfloor)$. The upper stripe corresponds to
    non-reversible Markov chains while the lower one to the reversible variants.}
  \label{fig:eig_comparison}
\end{figure}
Figure \ref{fig:eig_comparison} is based on 200.000 random Markov chains
with $n$ ranging from 54 to 2980 and with $k=\lfloor
\sqrt{n} \rfloor$. As we are interested in typical
behavior of these randomized Markov chains, we discarded the top and
bottom $5\%$ of the results for every $n$ considered.

The two type of Markov chains we compared can be seen as the extremal setups:
either the asymmetry is so strong that steps are deterministic along
the cycle, or we have perfect symmetry.
There is however a full spectrum of intermediate situations, and one may wonder which level of asymmetry is optimal.
We have seen that full
asymmetry is better in terms of mixing performance than full symmetry.
On Figure \ref{fig:probchange_comparison}, starting from the
reversible Markov chain, we gradually change the transition
probabilities along the cycle until we reach the current extreme
asymmetric case.
Specifically, for $1/2 \le q \le 1$ we set the transition matrix
$P_q = 
qP+(1-q)P^\top$ and compute the mixing rate of the resulting Markov
chain. Here we have $P_{1/2} = \tilde{P},~P_{1} = P$ as expected.
We perform simulations for $B_{500}(10)$ and $B_{500}(50)$.
In both case, 8000 random graphs were generated and the mixing rates were
computed for all graphs and for $q$ moving along $[1/2,1]$.
Again, the top and bottom
$5\%$ were discarded. The means of the resulting mixing rates are presented in Figure
\ref{fig:probchange_comparison} together with the sample standard
deviations. The figures show that the optimal choice is
near the extremal non-reversible case, confirming our concept.
Still, interestingly a minor offset
towards the reversible version still increases the mixing
rate. Intuitively the two effects of the modification match well: the
small loss in the speed of moving along the
cycle is well compensated by the local mixing introduced.
\begin{figure}[h]
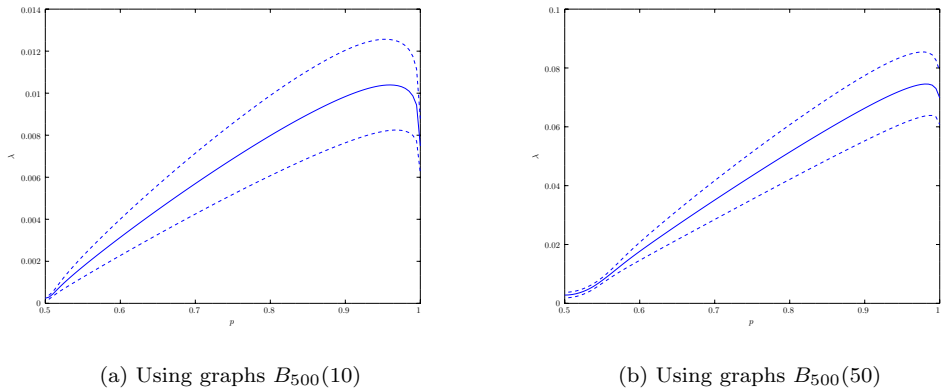

  \centering
  \subfloat[Using graphs $B_{500}(10)$]{
    %\resizebox{0.4\textwidth}{!}{\input{Images/cycle_probchange_summary_500_10}}
    \resizebox{0.4\textwidth}{!}{\input{cycle_probchange_summary_500_10}}
  }
  \subfloat[Using graphs $B_{500}(50)$]{
    %\resizebox{0.4\textwidth}{!}{\input{Images/cycle_probchange_summary_500_50}}
    \resizebox{0.4\textwidth}{!}{\input{cycle_probchange_summary_500_50}}
  }
  \caption{Mixing rates of the transition matrices $P_q = 
qP+(1-q)P^\top$ for the interpolation between the pure drift ($q=1$)
Markov chain and its reversible version ($q=0.5$). Solid lines
follow the means while dashed lines indicate the sample standard deviations
around the means.}
  \label{fig:probchange_comparison}
\end{figure}

The analytic treatment of the intermediate Markov chains for $q\neq
1/2,1$ brings new challenges as the non-reversible feature is still
present while we lose the deterministic nature of the movement along
the cycle.

\section{Conclusions}
\label{sec:conclusions}

We have seen in Theorem
\ref{cor:model1mixing} and Theorem \ref{thm:model2mixing} that for both models $B_n(k)$ and $B(L,k)$ the
mixing rate of the non-reversible Markov chain
considered is much higher than the similar reversible one.
The results confirm that the simultaneous application of adding long
distance edges and also setting the
Markov chain to be non-reversible dramatically improves the mixing rate.
We believe this phenomenon is promising and could provide similar
speedup effects for other reference graphs, other methods to add random
edges and other means of introducing non-reversibility.

We have also seen numerically in Figure \ref{fig:probchange_comparison}
that being fully
non-reversible is not necessarily optimal in this context, even though
it is significantly better than being fully reversible.

Therefore one of the open questions is to find the optimal
Markov chain among the intermediate cases. Another goal for future
research is to consider the more realistic situation where the hubs do
not have such high number of connection,
for instance, by replacing the complete graph on the selected $cn^{\sigma}$ nodes with
a random matching on them. Simulations similar to Figure
\ref{fig:eig_comparison} in \cite{gb:ringmixing2014} suggest that a
similar speedup is to be expected.

\bibliographystyle{siam}
\bibliography{ringmixing,swn}

\end{document}